\documentclass[10pt,a4paper]{article}
\usepackage[margin=1in]{geometry}
\usepackage{graphicx}
\usepackage[colorlinks=true, allcolors=black]{hyperref}
\usepackage{amsmath,amsthm,amsfonts,amssymb,amscd,mathrsfs}
\usepackage{enumerate}
\usepackage{mathrsfs}
\usepackage{xcolor}
\usepackage{graphicx}
\usepackage{hyperref}
\usepackage{enumitem}
\usepackage{bbm}
\usepackage{float}
\usepackage{fancyvrb}
\usepackage{booktabs}
\usepackage{subcaption}
\usepackage{algpseudocode,algorithmicx,algorithm}
\usepackage{etoolbox}
\patchcmd{\abstract}{\small}{}{}{}

\def\algbackskip{\hskip-\ALG@thistlm}
\newtheorem{theorem}{Theorem}[section]
\newtheorem{definition}[theorem]{Definition}

\newtheorem{lemma}[theorem]{Lemma}

\newtheorem{assumption}[theorem]{Assumption}
\newtheorem{problem}[theorem]{Problem}
\newtheorem*{theorem*}{Problem}
\theoremstyle{remark}
\newtheorem{remark}[theorem]{Remark}

\newcommand{\vecop}{\operatorname{vec}} 

\allowdisplaybreaks[4]
\title{A Convex Optimization Approach to Model-Free Inverse Optimal Control with Provable Convergence}
\date{\today}
\author{Meiling Yu\footnote{College of Artificial Intelligence, Nankai University,  China. Email: yumeiling0512@126.com}~~~~ Lechen Feng\footnote{College of Applied Mathematics, The Hong Kong Polytechnic University, China. Email: fenglechen0326@163.com}~~~~ Lei Jiang\footnote{China North Artificial Intelligence Innovation Research Institute, China. Email: jianglei@openloong.net}~~~~ Yuan-Hua Ni\footnote{College of Artificial Intelligence, Nankai University, China. Email: yhni@nankai.edu.cn}}
\begin{document}

\maketitle

\begin{abstract}
    Inverse Optimal Control (IOC) aims to infer the underlying cost functional of an agent from observations of its expert behavior. This paper focuses on the IOC problem within the continuous-time linear quadratic regulator framework, specifically addressing the challenging scenario where both the system dynamics and the cost functional weighting matrices are unknown. A significant limitation of existing methods for this joint estimation problem is the lack of rigorous theoretical guarantees on the convergence and convergence rate of their optimization algorithms, which restricts their application in safety-critical systems. To bridge this theoretical gap, we propose an analytical framework for IOC that provides such guarantees. The core contribution lies in the equivalent reformulation of this non-convex problem of jointly estimating system and cost parameters into a convex second-order cone programming problem. Building on this transformation, we design an efficient iterative solver based on the block successive upper-bound minimization algorithm. We rigorously prove that the proposed algorithm achieves a sublinear convergence rate of $\mathcal{O}(1/k)$. To the best of our knowledge, this is the first solution for the model-free IOC problem that comes with an explicit convergence rate guarantee. Finally, comparative simulation experiments against a state-of-the-art benchmark algorithm validate the superiority of our proposed method. The results demonstrate that our algorithm achieves an order-of-magnitude improvement in convergence speed while also exhibiting significant advantages in reconstruction accuracy and robustness.
\end{abstract}

\section{Introduction}\label{sec:introduction}
Optimal control has established a solid foundation for the success of modern control theory across a diverse range of fields, including robotics, aerospace, economics, and biological systems \cite{lewis2012optimal}. In particular, its theoretical cornerstone, the Linear Quadratic Regulator (LQR), has been fundamental to this success. The elegance of LQR framework lies in its ability to provide an analytical, optimal linear state-feedback control law when the system's linear dynamics $(A, B)$ and quadratic cost functional weighting matrices $(Q, R)$ are fully known. In many real-world applications, however, obtaining precise a priori model knowledge is often unattainable. This fundamental challenge has catalyzed the rapid development of Inverse Optimal Control (IOC) and Inverse Reinforcement Learning (IRL) \cite{ab2020inverse}\cite{adams2022survey}\cite{arora2021survey}\cite{chan2025inverse}. In contrast to optimal control, the core objective of IOC/IRL is to infer the latent cost functional driving an expert's observed behavior, thereby revealing the underlying intentions and preferences of the decision-making agent. This capability is of paramount importance for imitation learning, behavior prediction, and human-robot interaction.

Initial work in the field of IOC addresses the assumption that the system dynamics $(A, B)$ are known. In this classic setting, the problem is simplified to identifying the cost functional weighting matrices $(Q, R)$ from expert data. The seminal work \cite{kalman1964linear} first establishes the rigorous mathematical connection between the optimality of a linear system and its controller structure, providing a theoretical basis for all subsequent IOC research. Building on this, \cite{boyd1994linear} demonstrates that given a known optimal control law $K$, the problem of finding the corresponding cost functional can be solved by formulating a system of Linear Matrix Inequalities (LMIs) for the unknown matrices $Q$ and $R$. To better align with practical scenarios, subsequent research extends the model-known IOC problem to more complex settings, such as those involving uncertainty and noise (e.g., process and measurement noise) \cite{zhang2019inverse}\cite{zhang2022statistically}\cite{li2024inverse}\cite{karg2024bi}, partial state observability \cite{molloy2016discrete}, unknown control horizons \cite{qu2024control}, non-autonomous LQR \cite{jean2024inverse}, and multi-agent systems \cite{zhang2023inverse}\cite{hallinan2025inverse}.

A more challenging frontier, and one with greater practical relevance, is the scenario where both the system dynamics $(A, B)$ and the cost functional weighting matrices $(Q, R)$ are simultaneously unknown. This model-free setting substantially increases the problem's complexity, as it demands that an algorithm jointly reconstruct the system's dynamic properties and the expert's optimization objectives within a highly coupled framework. To address this challenge, the academic community has developed two main categories of data-driven methods.

The first class of methods comprises direct data-driven approaches, which are designed to bypass an explicit system identification step. For instance, \cite{qu20243dioc} leverages behavioral systems theory to propose an algorithm that directly recovers the objective functional from input-output trajectories, along with an analysis of its identification conditions and robustness. Drawing inspiration from reinforcement learning, the work \cite{xue2021inverse} introduces inverse reinforcement Q-learning, which learns the cost functional via methods like policy iteration, and successfully extends this concept to more complex scenarios such as static output feedback \cite{xue2023inverse} and antagonistic unknown systems \cite{lian2023off}. Moreover, to handle the stochasticity while reducing algorithmic complexity, the work \cite{clarke2022low} proposes an efficient low-complexity method that simultaneously recovers the control gain and cost parameters using the certainty equivalence principle and Semidefinite Programming (SDP) algorithm, providing an analysis of its solution properties. While these works highlight the ability to directly recover policies or cost functionals from data, their focus remains primarily on solvability, identifiability, or the introduction of novel learning paradigms.

The second category of methods is built upon general-purpose differentiable programming and optimization frameworks. The Pontryagin Differentiable Programming (PDP) framework \cite{jin2020pontryagin}\cite{jin2021safe} stands as a representative work in this area. By differentiating through Pontryagin's Maximum Principle, it provides a powerful end-to-end learning and control platform capable of handling complex dynamics and objective functionals. Similarly, researchers have proposed IRL frameworks based on Differential Dynamic Programming (DDP) \cite{cao2024differential} and trust-region methods \cite{cao2023trust}. These frameworks excel at efficiently computing policy gradients to learn system and cost parameters and possess strong versatility for high-dimensional problems. Yet, they do not provide a deep analysis or guarantees regarding the convergence properties (e.g., convergence rate, attainment of a global optimum) of the gradient-based optimization process itself.

Despite these significant advancements, a critical theoretical gap persists: for iterative algorithms that jointly estimate system dynamics parameters $(A, B)$ and cost functional parameters $(Q, R)$, a rigorous convergence analysis, particularly one that guarantees a specific convergence rate, remains largely unexplored. Existing literature has focused on the existence and identifiability of solutions or the efficiency of gradient computations, while often neglecting rigorous proof of the convergence properties of the solution algorithm itself. This absence of theoretical guarantees limits the applicability of these algorithms in high-reliability settings, such as aerospace and medical systems, where system and cost parameters are highly coupled, and the resulting optimization landscape may be non-convex and complex.

To fill this critical theoretical void, this paper is dedicated to developing an analytical framework for the model-free IOC problem that, by leveraging the structural properties of LQR, not only jointly estimates the system dynamics and cost functional but also provides rigorous guarantees on the convergence and convergence rate of its solution algorithm. Our specific contributions are as follows.
\begin{enumerate}
    \item For the problem where both the system model and cost functional are unknown, we construct an IOC framework that relies on a single optimal expert trajectory. We systematically reformulate this complex joint estimation problem from a non-convex optimization problem into an equivalent convex one. This transformation guides a seemingly intractable problem into a well-structured, theoretically mature convex optimization framework, establishing a solid foundation for subsequent algorithm design and convergence analysis.

    \item For the resulting convex problem, we address its Lagrangian dual rather than the primal form. We design a highly efficient solver based on the Block Successive Upper-bound Minimization (BSUM) algorithm \cite{hong2017iteration}\cite{feng2024two}. The key advantage of this method lies in its ability to leverage the problem's block structure, decomposing the joint optimization into multiple independent subproblems. We prove that each of these subproblems admits an analytical, closed-form solution, thereby avoiding costly inner-loop iterations and ensuring computational efficiency. Crucially, we provide a rigorous convergence analysis for the proposed BSUM algorithm, proving that the algorithm achieves an explicit sublinear convergence rate of $\mathcal{O}(1/k)$.

    \item Finally, the efficiency and superiority of the algorithm are validated through simulation experiments. We conduct a systematic comparison of our algorithm against the PDP algorithm. The experimental results demonstrate that our proposed method not only achieves an order-of-magnitude improvement in convergence speed but also shows significant advantages in reconstruction accuracy and robustness, thereby proving the practical value of our theoretical framework.
\end{enumerate}

The remainder of this paper is organized as follows. Section \ref{sec-pre} first introduces the dynamics of continuous-time linear time-invarian systems and the LQR framework. Building upon this foundation, we formally construct the model-free inverse optimal control problem of recovering system dynamics and the cost functional from a single expert trajectory. Section \ref{main-results} constitutes the theoretical core of our work. Here, we equivalently reformulate the non-convex inverse problem into a convex second-order cone programming problem. Subsequently, an efficient solution algorithm based on its dual form is designed, and we rigorously prove that this algorithm achieves a sublinear convergence rate of $\mathcal{O}(1/k)$. In Section \ref{sim}, a series of simulation experiments are conducted to systematically validate the convergence and stability of the proposed algorithm. Finally, Section \ref{conclusion} concludes the paper with a summary of our findings and a discussion of potential directions for future research.

\textit{Notation.}
In this paper, $\mathbb{R}^n$ denotes the space of $n$-dimensional real vectors, and $\mathbb{R}^{n \times m}$ denotes the space of $n \times m$ real matrices. $I_n$ and $\mathbf{0}_n$ represent the $n \times n$ identity and zero matrices, respectively (subscripts may be omitted when the dimensions are clear from the context). For a matrix $A$, its transpose, inverse, and Moore-Penrose pseudoinverse are denoted by $A^\top$, $A^{-1}$, and $A^\dagger$, respectively. The Euclidean inner product between vectors $x, y \in \mathbb{R}^n$ is denoted by $\langle x, y \rangle = x^\top y$. The set of $n \times n$ symmetric matrices is denoted by $\mathbb{S}^{n}$. The sets of $n \times n$ symmetric positive definite and positive semidefinite matrices are denoted by $\mathbb{S}_{++}^n$ and $\mathbb{S}_{+}^n$, respectively. For symmetric matrices $A, B$, the notation $A \succ B$ (or $A \succeq B$) indicates that the matrix $A-B$ is positive definite (or positive semidefinite). In particular, $A \succ \mathbf{0}$ signifies that $A$ is a positive definite matrix. For a symmetric matrix $A$, $\lambda_{\text{max}}(A)$ denotes its maximum eigenvalue. The symbol $\otimes$ denotes the Kronecker product, and $\vecop(\cdot)$ denotes the vectorization operator. $\Gamma_{++}^k$ and $\Gamma_{+}^k$ represent the convex cones formed by the vectorization of $k \times k$ symmetric positive definite and positive semidefinite matrices, respectively. Finally, $I_C(\cdot)$ is the indicator function for any subset $C\in \mathbb{R}^n$.

\section{Preliminaries}\label{sec-pre}

Consider a continuous-time linear time-invariant (LTI) system described by the state-space equation:
\begin{equation} \label{eq:dynamics}
\dot{x}(t) = Ax(t) + Bu(t),
\end{equation}
where $x(t) \in \mathbb{R}^n$ is the system's state vector and $u(t) \in \mathbb{R}^m$ is the control input vector. The system matrices $A \in \mathbb{R}^{n \times n}$ and $B \in \mathbb{R}^{n \times m}$ are unknown. We pose the following assumption throughout the paper.

\begin{assumption}\label{assump_0}
The true system pair $(A, B)$ is assumed to be stabilizable and the pair $(A,\sqrt{Q})$ is detectable.
\end{assumption}

We presume the existence of an unknown infinite-horizon quadratic cost functional $J$, which is minimized by the optimal control law:
\begin{equation} \label{eq:cost}
J(u) = \int_{0}^{\infty} \left[ x(t)^\top Q x(t) + u(t)^\top R u(t) \right] dt,
\end{equation}
where $Q \in \mathbb{S}_{+}^{n}$ is the unknown state weighting matrix and $R \in \mathbb{S}_{++}^{m}$ is the unknown control weighting matrix.

\begin{assumption}\label{assump_1}
The observed system trajectory $x^*(\cdot)$ is generated by an optimal controller $u^*(\cdot)$ that minimizes the unknown cost functional $J(u)$ \eqref{eq:cost} for the unknown system dynamics \eqref{eq:dynamics}. Furthermore, this trajectory satisfies the Persistent Excitation (PE) condition, meaning it is sufficiently rich in information to allow for the unique identification of the expert's underlying linear feedback law from the observed data.
\end{assumption}
\begin{remark}
    The PE assumption is reasonable in practice as, while the core framework of this paper is deterministic, real-world physical systems are commonly subject to process noise. These stochastic disturbances continuously excite the system, ensuring that the state trajectory can be sufficient to satisfy the PE condition even during stabilization tasks.
\end{remark}

According to standard LQR theory, the optimal control law is a linear state feedback:
\begin{equation} \label{eq:control_law}
u(t) = -K^* x(t),\quad t\in[0,\infty),
\end{equation}
where $K^* \in \mathbb{R}^{m \times n}$ is the optimal feedback gain matrix, given by
\begin{equation} \label{eq:gain}
K^* = R^{-1} B^\top P.
\end{equation}
Here, $P \in \mathbb{S}_{++}^{n}$ is the unique symmetric positive definite solution to the corresponding Continuous-time Algebraic Riccati Equation (ARE):
\begin{equation} \label{eq:are}
A^\top P + P A - P B R^{-1} B^\top P + Q = \mathbf{0}_{n}.
\end{equation}
Furthermore, the resulting closed-loop system $\dot{x}(t) = (A - BK^*)x(t)$ is asymptotically stable.

The inverse optimal control problem relies on observed expert data. Typically, this data comprises one or more trajectories of states and their corresponding optimal control inputs:
\begin{equation}
\mathcal{D}_N = \left\{ (x_i^*(t), u_i^*(t)) \mid t \in [0, T_i], i = 1, \dots, N \right\}.
\end{equation}
These trajectories are assumed to be generated under the action of the optimal feedback matrix $K^*$ associated with the unknown $(A, B, Q, R)$, corresponding to different initial conditions $x_i(0)$. The observed data must satisfy the system dynamics \eqref{eq:dynamics} and the optimality condition \eqref{eq:are}. This paper primarily considers the scenario where only a single optimal trajectory is available.

\begin{problem}[Model-Free IOC]\label{problem_1}
Given a state-input trajectory $\mathcal{D}=\left\{ (x^*(t), u^*(t)) \mid t \in [0, T]\right\}$ generated by an optimal controller for an unknown LTI system $(A, B)$ and an unknown cost functional $(Q, R)$, the objective is to find a set of estimated parameters $(\hat{A}, \hat{B}, \hat{Q}, \hat{R})$. Under these parameters, the trajectory $(\hat{x}(\cdot), \hat{u}(\cdot))$ generated by the corresponding optimal controller should best reproduce the expert's given trajectory $(x^*(\cdot), u^*(\cdot))$. This problem can be formulated as the following optimization problem
\begin{equation} \label{eq:ioc_opt_formulation}
    \begin{aligned}
    & \min_{\hat{A}, \hat{B}, \hat{Q}, \hat{R}} &&  \int_{0}^{T} \lVert \hat{x}(t) - x^*(t) \rVert_2^2+\lVert \hat{u}(t) - u^*(t) \rVert_2^2 dt\\
    & \text{\quad s.t.} && \dot{\hat{x}}(t)=\hat{A}\hat{x}(t)+\hat{B}\hat{u}(t),\\
    & && (\hat{x}(\cdot),\hat{u}(\cdot))\in \arg\min \int_0^\infty \left[ x(t)^\top \hat{Q} x(t) + u(t)^\top \hat{R} u(t) \right] dt.
    \end{aligned}
\end{equation}
\end{problem}



\begin{remark}
Our framework is built upon Assumption \ref{assump_1}, which presumes that the expert behavior originates from an ideal LQR system. In practice, this assumption may not hold perfectly due to system nonlinearities, complex decision logic, or sub-optimality in the expert's behavior. The advantage of our framework is that it does not compel the discovery of a "true" system that perfectly matches the data. Instead, it aims to find the LQR model that best approximates the expert's behavior. Consequently, even when discrepancies exist between the expert's behavior and any ideal LQR model, our method can still provide the closest and most interpretable linear-quadratic approximation.
\end{remark}

Non-uniqueness is an intrinsic property of the model-free IOC problem \cite{town2025nonuniqueness}. A classic example is that scaling the cost functional weighting matrices $(Q, R)$ by an arbitrary positive scalar does not alter the final optimal control law $K^*$ and, therefore, does not change the expert's trajectory. To formally address the ambiguity arising from different parameter sets yielding identical behavior, we introduce the definition of equivalent solutions.

\begin{definition}[Equivalent Solutions]
Within the framework of the model-free IOC problem (Problem \ref{problem_1}), consider two distinct sets of system and cost functional parameters, $S_1 = (\hat{A}_1, \hat{B}_1, \hat{Q}_1, \hat{R}_1)$ and $S_2 = (\hat{A}_2, \hat{B}_2, \hat{Q}_2, \hat{R}_2)$. If both sets of parameters satisfy all the conditions defined in Problem \ref{problem_1} and their resulting optimal state and control trajectories, $(\hat{x}(\cdot),\hat{u}(\cdot))$, are observationally indistinguishable, then $S_1$ and $S_2$ are termed equivalent solutions.
\end{definition}


Accordingly, the objective of this research is not to identify the unique "true" parameters from an infinite space of possibilities. Instead, our goal is to find any representative element from the set of equivalent solutions. The optimization framework proposed in this paper is designed precisely to find this equivalence set. All solutions within this set, despite their potentially different numerical parameter values, can perfectly explain the expert's behavior. The convergence guarantee of our algorithm, therefore, refers to its ability to reliably converge to a feasible point within this set of equivalent solutions.

\section{Main Results}\label{main-results}

\subsection{Optimization Formulation for Model-Free Infinite-Horizon IOC}
As previously discussed, the model-free Inverse Optimal Control (IOC) problem requires the simultaneous reconstruction of the system dynamics matrices $(A, B)$ and the cost functional weighting matrices $(Q, R)$ from an expert trajectory. However, these unknown parameters are deeply coupled within the Algebraic Riccati Equation (ARE), and the problem itself is non-convex, which makes obtaining a unique solution via direct analytical methods extremely difficult, if not infeasible.
To address this challenge, this paper proposes a new solution framework based on the key premise established in Assumption \ref{assump_1}, that the expert trajectory is not only optimal but also satisfies the PE condition. According to system identification theory, the PE condition is critical for ensuring that the linear feedback gain $K^*$ can be uniquely identified by solving the overdetermined system of linear equations $u^*(t) = -K^*x^*(t)$ (e.g., via least squares) \cite{ASTROM1971123}.
This step transforms $K^*$ from an unknown variable into a known optimization target. Consequently, the core task of the inverse problem is reframed as: finding a set of system and cost parameters $(\hat{A}, \hat{B}, \hat{Q}, \hat{R})$ that is compatible with this identified gain $K^*$. Here, the compatibility requires that the estimated parameter set satisfy the LQR optimality conditions and that the resulting optimal behavior (i.e., the optimal gain $\hat{K}$ and trajectory $\hat{x}(\cdot)$) aligns with the expert's behavior $(x^*(\cdot), K^*)$ in an optimal sense.

Based on this principle, we construct the following optimization framework to jointly estimate the system and cost functional parameters
\begin{equation} \label{eq:ioc_opt_formulation}
    \begin{aligned}
    & \min_{\hat{A}, \hat{B}, \hat{Q}, \hat{R}, \hat{P}} && \lVert \hat{K} - K^* \rVert_2^2 +  \int_{0}^{T} \lVert \hat{x}(t) - x^*(t) \rVert_2^2 dt\\
    & \text{\quad { }s.t.} && \hat{A}^\top \hat{P} + \hat{P}\hat{A} - \hat{P}\hat{B}\hat{R}^{-1}\hat{B}^\top \hat{P} + \hat{Q} = \mathbf{0}, \\
    &             && \hat{K} = \hat{R}^{-1} \hat{B}^\top \hat{P}, \\
    &             && \dot{\hat{x}}(t) =e^{\hat{A}_Kt}\hat{x}(0), \quad \hat{x}(0) = x^*(0), \\
    &             && \hat{Q} = \hat{Q}^\top \succeq \mathbf{0}, \\
    &             && \hat{R} = \hat{R}^\top \succeq \epsilon I, \\
    &             && \hat{P} = \hat{P}^\top \succeq \epsilon I,
    \end{aligned}
\end{equation}
where $\epsilon>0$ and $\hat{A}, \hat{B}, \hat{Q}, \hat{R}$ are the unknown matrices to be optimized.  $\hat{P}$ is the associated ARE solution, and $\hat{K}$ is the corresponding optimal feedback gain.  $\hat{x}(t)$ is the simulated trajectory generated by the estimated parameters $(\hat{A}, \hat{B}, \hat{K})$ from the initial condition $x^*(0)$, and $\hat{A}_K=\hat{A}-\hat{B}\hat{K}$ . 

\begin{remark}
Standard LQR theory requires the control weighting matrix $R$ and the Riccati solution $P$ to be strictly positive definite. However, the set of positive definite matrices constitutes an open convex cone, and the projection operator onto an open set is not well-defined. To ensure our optimization algorithm, particularly the projection steps required by our solver, is computationally well-posed, we relax the strict positive definiteness constraints. Specifically, constraints of the form $\hat{X} \succ \mathbf{0}$ are relaxed to $\hat{X} \succeq \epsilon I$ for a small, user-defined positive scalar $\epsilon$ (e.g., $\epsilon=10^{-6}$). This common technique transforms the constraint set into a closed convex cone, upon which the projection is a standard and tractable operation, while practically enforcing near-strict positive definiteness.
\end{remark}

\begin{lemma}\label{lemma_1}
    Letting Assumptions \ref{assump_0} and \ref{assump_1} hold, the optimization problem \eqref{eq:ioc_opt_formulation} is well-defined. 
\end{lemma}
\begin{proof}
To prove that problem \eqref{eq:ioc_opt_formulation} is well-defined, we must show that its feasible set is non-empty and that an optimal solution exists. 
Under Assumption \ref{assump_1}, the observed data is generated by a true set of system and cost parameters $(A^*, B^*, Q^*, R^*)$.  According to LQR theory, this parameter set corresponds to a unique symmetric positive definite ARE solution $P^*$, which yields the optimal control gain $K^* = (R^*)^{-1}(B^*)^{\top}P^*$, and the resulting closed-loop system is stable. 
We construct a candidate solution by letting $(\hat{A}, \hat{B}, \hat{Q}, \hat{R}, \hat{P}) = (A^*, B^*, Q^*, R^*, P^*)$.  We verify the feasibility of this candidate. First, this tuple inherently satisfies the ARE constraint.  Second, the computed gain $\hat{K} = (R^*)^{-1}(B^*)^{\top}P^*$ is precisely the expert's optimal gain $K^*$.  Likewise, since both the system parameters and the control gain are identical to their true values, the trajectory $\hat{x}(\cdot)$ generated from the same initial state will be identical to the observed trajectory $x^*(\cdot)$. Finally, $Q^*$ is symmetric positive semidefinite, while $R^*$ and $P^*$ are symmetric positive definite. Therefore, this candidate solution satisfies all the constraints of problem \eqref{eq:ioc_opt_formulation}, rendering its feasible set non-empty. 

Next, we examine the objective functional value for this feasible solution. The objective functional of problem \eqref{eq:ioc_opt_formulation}, $J(\cdot) = \lVert \hat{K} - K^* \rVert_2^2 + \int_{0}^{T} \lVert \hat{x}(t) - x^*(t) \rVert_2^2 dt$, is clearly non-negative and has a theoretical lower bound of 0.  Substituting the feasible solution $(A^*, B^*, Q^*, R^*, P^*)$ into the objective functional yields
$$J(A^*, B^*, Q^*, R^*, P^*) = \lVert K^* - K^* \rVert_2^2 + \int_{0}^{T} \lVert x^*(t) - x^*(t) \rVert_2^2 dt = 0.$$
Since this feasible solution achieves the theoretical lower bound of the objective functional, $(A^*, B^*, Q^*, R^*, P^*)$ is an optimal solution to problem \eqref{eq:ioc_opt_formulation}. 
In summary, because an optimal solution exists, we can conclude that optimization problem \eqref{eq:ioc_opt_formulation} is well-defined. 
\end{proof}

\begin{lemma}\label{lemma_2}
    If $A^i x_0 = \tilde{A}^i x_0$ holds for all $i = 0, 1, \dots, n$, then the two linear systems $\dot{x} = Ax$ and $\dot{x} = \tilde{A}x$ produce identical solution trajectories from the same initial state $x(0) = x_0$. 
\end{lemma}
\begin{proof}
    We need to prove that if $A^i x_0 = \tilde{A}^i x_0$ for $i = 0, 1, \dots, n$, then $e^{At} x_0 = e^{\tilde{A}t} x_0$.  Since $e^{Mt} = \sum_{k=0}^{\infty} \frac{t^k}{k!} M^k$, this is equivalent to proving that $A^k x_0 = \tilde{A}^k x_0$ for all integers $k \ge 0$.  We proceed by induction. The base cases for $k = 0, \dots, n$ are given by the lemma's premise.
    
    By the Cayley-Hamilton theorem, matrix $A$ satisfies its own characteristic polynomial $p(\lambda) = \det(\lambda I - A) = \lambda^n + a_{n-1} \lambda^{n-1} + \dots + a_0$.  Thus, $p(A) = A^n + \sum_{i=0}^{n-1} a_i A^i = \mathbf{0}$. Applying this to the vector $x_0$, we obtain
\begin{equation}
    \label{eq:pic_proof_1}
    A^{n} x_0 + a_1 A^{n-1} x_0 + \dots + a_{n-1} Ax_0 + a_n x_0 = \mathbf{0}.
\end{equation}
Since we have $A^k x_0 = \tilde{A}^k x_0$ for $k = 0, 1, \dots, n$, it follows that
\begin{equation} \label{eq:pic_proof_2}
\tilde{A}^n x_0 + a_1 \tilde{A}^{n-1} x_0 + \dots + a_n x_0 = \mathbf{0}.
\end{equation}
Left-multiplying equation \eqref{eq:pic_proof_2} by $\tilde{A}$ yields
\[ 
\tilde{A}^{n+1} x_0 + a_1 \tilde{A}^n x_0 + \dots + a_n \tilde{A} x_0 = \mathbf{0}, 
\]
which can be rearranged to
\[ 
\tilde{A}^{n+1} x_0 = - \left[ a_1 \tilde{A}^n x_0 + \dots + a_n \tilde{A} x_0 \right]. 
\]
Again using the fact that $A^i x_0 = \tilde{A}^i x_0$ holds for $i=1, \dots, n$, we have
\[ 
\tilde{A}^{n+1} x_0 = - \left[ a_1 A^n x_0 + \dots + a_n A x_0 \right]. 
\]
For comparison, let us left-multiply equation \eqref{eq:pic_proof_1} by $A$
\[ 
A^{n+1} x_0 + a_1 A^n x_0 + \dots + a_n A x_0 = \mathbf{0}, 
\]
which gives
\[ 
A^{n+1} x_0 = - \left[ a_1 A^n x_0 + \dots + a_n A x_0 \right]. 
\]
By comparing the expressions for $A^{n+1} x_0$ and $\tilde{A}^{n+1} x_0$, we conclude that $A^{n+1} x_0 = \tilde{A}^{n+1} x_0$. By induction, it can be shown that $A^k x_0 = \tilde{A}^k x_0$ holds for all $k \ge 0$. Therefore,
\[ 
e^{At} x_0 = \sum_{k=0}^{\infty} \frac{t^k}{k!} A^k x_0 = \sum_{k=0}^{\infty} \frac{t^k}{k!} \tilde{A}^k x_0 = e^{\tilde{A}t} x_0.
\]   
    This completes the proof. 
\end{proof}

Given that the observed system trajectory is $x^*(t)$, which is generated by an unknown LTI system $\dot{x}(t) = A_K^* x(t)$ where $A_K^* = A - BK^*$, the derivatives of the trajectory satisfy the relations
\begin{align*}
\dot{x}^*(t) &= A_K^* x^*(t), \\
\ddot{x}^*(t) &= A_K^* \dot{x}^*(t), \\
&\vdots \\
x^{*(n)}(t) &= A_K^* x^{*(n-1)}(t),
\end{align*}
where $x^{*(k)}(t)$ denotes the $k$-th time derivative of $x^*(t)$. 
To estimate $A_K^*$ from these relations, we can sample the trajectory and its derivatives at $N$ time points $t_1, t_2, \dots, t_N$.  We define the data matrices
\begin{align*}
\Lambda_i &= \begin{bmatrix} x^{*(i)}(t_1) & x^{*(i)}(t_2) & \dots & x^{*(i)}(t_N) \end{bmatrix} \in \mathbb{R}^{n \times N},
\end{align*}
for $i=0, \dots, n$.  Based on the derivative relations, we have
\begin{align*}
\Lambda_1 = A_K^* \Lambda_0, \quad \Lambda_2 = A_K^*\Lambda_1, \quad \dots, \quad \Lambda_n = A_K^* \Lambda_{n-1}.
\end{align*}
These equations can be consolidated into a larger matrix equation.  Let
\begin{align*}
\bar{\Lambda}_1 &:= \begin{bmatrix} \Lambda_0 & \Lambda_1 & \dots & \Lambda_{n-1} \end{bmatrix} \in \mathbb{R}^{n \times (nN)}, \\
\bar{\Lambda}_2 &:= \begin{bmatrix} \Lambda_1 & \Lambda_2 & \dots & \Lambda_{n} \end{bmatrix} \in \mathbb{R}^{n \times (nN)}.
\end{align*}
Then, the series of derivative relations can be written as a single matrix equation
\begin{equation} \label{eq:ak_lambda_relation}
A_K^* \bar{\Lambda}_1 = \bar{\Lambda}_2.
\end{equation}


Lemma \ref{lemma_1} has already established that the theoretical optimal value of problem (\ref{eq:ioc_opt_formulation}) is zero.  This conclusion implies that finding an optimal solution is equivalent to finding a feasible solution for which the objective value is zero.  Building on this insight, to simplify the problem and reveal its inherent convex structure, we decouple the non-linear coupled constraints by introducing an auxiliary variable $Z=\hat{A}^\top\hat{P}$ and the closed-loop matrix $\hat{A}_K=\hat{A}-\hat{B}\hat{K}$.  This allows us to reformulate the problem as the following feasibility problem
\begin{equation} \label{eq:ioc_feasibility_form}
\begin{aligned}
&\min_{Z, \hat{Q}, \hat{R}, \hat{P}}  && 0 \\
&\text{\quad s.t.}  && Z^\top + Z - {K^*}^\top \hat{R} K^* + \hat{Q} = \mathbf{0}, \\
& && Z^\top - {K^*}^\top \hat{R} K^* = \hat{P} \hat{A}_K, \\
& && \hat{A}_K \bar{\Lambda}_1 = \bar{\Lambda}_2,\\
& && \hat{Q} = \hat{Q}^\top \succeq \mathbf{0}, \\
& && \hat{R} = \hat{R}^\top \succeq \epsilon I, \\
& && \hat{P} = \hat{P}^\top \succeq \epsilon I.
\end{aligned}
\end{equation}
The following theorem precisely establishes the mathematical relationship between the original optimization problem (\ref{eq:ioc_opt_formulation}) and this feasibility problem (\ref{eq:ioc_feasibility_form}). To facilitate its proof, we begin by presenting an essential lemma.
\begin{lemma} \label{lemma:trajectory_identity}
    Let $(Z,\hat{Q},\hat{R},\hat{P})$ be a feasible solution to the problem~\eqref{eq:ioc_feasibility_form}, which satisfies $\hat{A}_K \bar{\Lambda}_{1} = \bar{\Lambda}_{2}$. If the data matrix $\Lambda_0$ is constructed to include the initial state $x^*(0)$ as one of its column vectors, then the trajectory $\hat{x}(\cdot)$ generated by the closed-loop dynamics $\dot{x} = \hat{A}_K x$ with the initial condition $\hat{x}(0)=x^*(0)$ is identical to the expert trajectory $x^*(\cdot)$.
\end{lemma}
\begin{proof}
The condition $\hat{A}_K \bar{\Lambda}_{1} = \bar{\Lambda}_{2}$, combined with the expert system's dynamics $A_K^* \bar{\Lambda}_{1} = \bar{\Lambda}_{2}$, yields the equality $\hat{A}_K \bar{\Lambda}_{1} = A_K^* \bar{\Lambda}_{1}$. This matrix equality implies that $\hat{A}_K \Lambda_i = A_K^* \Lambda_i$ for each data block, where $i=0, \dots, n-1$. By combining this with the definition of the data matrices, $\Lambda_i=A_K^* \Lambda_{i-1} = (A_K^*)^i \Lambda_0$, where $i=1,\dots,n$, it follows that
\begin{equation}
    (\hat{A}_K)^k \Lambda_0 = (A_K^*)^k \Lambda_0, \quad \text{for all } k=1, \dots, n.\nonumber
\end{equation}
Since we specify that the data matrix $\Lambda_0$ contains the initial state $x^*(0)$ as one of its column vectors, the above matrix equality must hold for this specific column. Therefore, we have
\begin{equation}
    (\hat{A}_K)^k x^*(0) = (A_K^*)^k x^*(0), \quad \text{for } k=1, \dots, n.\nonumber
\end{equation}
This result covers all cases required by Lemma~\ref{lemma_2} for $k=0, \dots, n$ (as the case for $k=0$ is trivially true). Thus, the conditions for Lemma~\ref{lemma_2} are fully satisfied. It follows that the generated trajectory $\hat{x}(\cdot)$ is identical to the expert trajectory $x^*(\cdot)$.
\end{proof}
\begin{theorem}
Let $\Phi^*$ be the set of optimal solutions to the parameter optimization problem \eqref{eq:ioc_opt_formulation}, and let $\Psi_{\text{feas}}$ be the feasible set of the feasibility problem \eqref{eq:ioc_feasibility_form}.  Under Assumption \ref{assump_1}, these two sets are equivalent, i.e.,
$$\Phi^* = \Psi_{\text{feas}}.$$
\end{theorem}
\begin{proof}
We prove the equivalence by demonstrating mutual inclusion. 

($\Phi^* \subseteq \Psi_{\text{feas}}$). First, we show that any optimal solution to problem \eqref{eq:ioc_opt_formulation} corresponds to a feasible solution of problem \eqref{eq:ioc_feasibility_form}. 
Let $(\hat{A},\hat{B},\hat{Q},\hat{R},\hat{P}) \in \Phi^*$ be an arbitrary optimal solution.  According to Lemma \ref{lemma_1}, the minimum objective value of problem \eqref{eq:ioc_opt_formulation} is 0, so for this optimal solution, the objective value must be 0.  This leads directly to two conclusions: (i) the optimal gain matches the expert's gain, i.e., $\hat{K} = K^*$, where $\hat{K}=\hat{R}^{-1}\hat{B}^{\top}\hat{P}$; and (ii) the generated trajectory perfectly matches the expert's trajectory, i.e., $\hat{x}(\cdot) = x^{*}(\cdot)$, which implies that the closed-loop dynamics matrix $\hat{A}_K = \hat{A}-\hat{B}K^*$ satisfies the trajectory consistency condition $\hat{A}_{K}\bar{\Lambda}_{1}=\bar{\Lambda}_{2}$. 

From this optimal solution $(\hat{A},\hat{B},\hat{Q},\hat{R},\hat{P})$, we construct a candidate solution for problem \eqref{eq:ioc_feasibility_form} as follows.  Let the auxiliary variable be $Z = \hat{A}^{\top}\hat{P}$, and retain $\hat{Q}, \hat{R}, \hat{P},$ and $\hat{A}_K$ from the solution.  We now verify that this candidate $(Z, \hat{Q}, \hat{R}, \hat{P})$ satisfies all constraints of problem \eqref{eq:ioc_feasibility_form}. 
(i) For the constraint $Z^{\top}+Z-{K^{*}}^{\top}\hat{R}K^{*}+\hat{Q}=0$: Substituting $Z = \hat{A}^{\top}\hat{P}$ and using the identity ${K^{*}}^{\top}\hat{R}K^{*} = \hat{P}\hat{B}\hat{R}^{-1}\hat{B}^{\top}\hat{P}$ derived from $\hat{K}=K^*$, the constraint can be rewritten as $\hat{P}\hat{A} + \hat{A}^{\top}\hat{P} - \hat{P}\hat{B}\hat{R}^{-1}\hat{B}^{\top}\hat{P} + \hat{Q} = 0$, which is precisely the ARE from problem \eqref{eq:ioc_opt_formulation}.  Thus, this constraint is satisfied.
(ii) For the constraint $Z^{\top}-{K^{*}}^{\top}\hat{R}K^{*}=\hat{P}\hat{A}_{K}$: Substituting $Z^\top = \hat{P}\hat{A}$ and using the relation ${K^{*}}^{\top}\hat{R}K^{*} = \hat{P}\hat{B}K^*$ (derived from $K^*=\hat{R}^{-1}\hat{B}^{\top}\hat{P}$), the constraint becomes $\hat{P}\hat{A} - \hat{P}\hat{B}K^* = \hat{P}\hat{A}_K$, which simplifies to $\hat{P}(\hat{A} - \hat{B}K^*) = \hat{P}\hat{A}_K$.  Since $\hat{A}_K=\hat{A}-\hat{B}K^*$, this constraint is naturally satisfied. 
(iii) The remaining constraints regarding trajectory consistency and matrix definiteness are inherent properties of the optimal solution as shown above, and are thus satisfied. 
In conclusion, any element in $\Phi^*$ corresponds to a feasible solution in $\Psi_{\text{feas}}$, so $\Phi^* \subseteq \Psi_{\text{feas}}$. 

($\Psi_{\text{feas}} \subseteq \Phi^*$). Conversely, we now show that any feasible solution to problem \eqref{eq:ioc_feasibility_form} corresponds to an optimal solution of problem \eqref{eq:ioc_opt_formulation}. 
Take an arbitrary feasible solution $(Z, \hat{Q}, \hat{R}, \hat{P}) \in \Psi_{\text{feas}}$.  We construct a candidate solution $(\hat{A},\hat{B},\hat{Q},\hat{R},\hat{P})$ for problem \eqref{eq:ioc_opt_formulation}, where $\hat{Q}, \hat{R}, \hat{P}$ are taken directly from the feasible solution, and $\hat{A}$ and $\hat{B}$ are defined as
$$\hat{A} = \hat{P}^{-1}Z^\top, \quad \hat{B} = \hat{P}^{-1}{K^{*}}^{\top}\hat{R}.$$
Since $\hat{P} \succ \mathbf{0}$, its inverse exists, and the above definitions are well-posed. 
We first demonstrate that this candidate solution is feasible for problem \eqref{eq:ioc_opt_formulation}.  Through algebraic manipulation, it can be verified that under these definitions of $\hat{A}$ and $\hat{B}$, the ARE constraint of problem \eqref{eq:ioc_opt_formulation} is equivalent to the first constraint of problem \eqref{eq:ioc_feasibility_form} and is therefore satisfied.  For the gain constraint, we compute $\hat{K}=\hat{R}^{-1}\hat{B}^{\top}\hat{P} = \hat{R}^{-1}(\hat{R}K^*\hat{P}^{-1})\hat{P} = K^*$, which shows that $\hat{K}=K^*$.  The definiteness constraints are directly guaranteed by the definition of $\Psi_{\text{feas}}$. 
Next, we prove the optimality of this candidate solution.  The objective functional of problem \eqref{eq:ioc_opt_formulation} is $J = ||\hat{K}-K^{*}||_{2}^{2}+\int_{0}^{T}||\hat{x}(t)-x^{*}(t)||_{2}^{2}dt$.  We have already shown that $\hat{K} = K^*$, so the first term of the objective is zero.  For the second term, since $\hat{A}_K$ satisfies the third constraint of problem \eqref{eq:ioc_feasibility_form}, $\hat{A}_{K}\bar{\Lambda}_{1}=\bar{\Lambda}_{2}=A^*_K\bar{\Lambda}_{1}$, and the data matrix $\Lambda_0$ contains the initial state$x^*(0)$, we can directly apply Lemma \ref{lemma:trajectory_identity} to conclude that the generated trajectory $\hat{x}(\cdot)$ is identical to the expert trajectory $x^{*}(\cdot)$. Thus, the second term of the objective is also zero.  Since this candidate solution yields an objective value of zero, the theoretical minimum, it is an optimal solution to problem \eqref{eq:ioc_opt_formulation} and belongs to the set $\Phi^*$.  Therefore, $\Psi_{\text{feas}} \subseteq \Phi^*$. 

Having shown that $\Phi^* \subseteq \Psi_{\text{feas}}$ and $\Psi_{\text{feas}} \subseteq \Phi^*$, we conclude that the two sets are equal. 
\end{proof}

By introducing an auxiliary variable $G = \hat{P} \hat{A}_K$ and left-multiplying the constraint $\hat{A}_K \bar{\Lambda}_1 = \bar{\Lambda}_2$ by $\hat{P}$, we obtain $G\bar{\Lambda}_1 = \hat{P}\bar{\Lambda}_2$. This allows the optimization problem to be further reformulated as
\begin{equation} \label{eq:ioc_feasibility_form_2}
\begin{aligned}
& \min_{Z, \hat{Q}, \hat{R}, \hat{P},G} && 0 \\
&\text{\quad { } s.t.}  && Z^\top + Z - {K^*}^\top \hat{R} K^* + \hat{Q} = \mathbf{0}, \\
& && Z^\top - {K^*}^\top \hat{R} K^* = G,\\
& && G\bar{\Lambda}_1 = \hat{P}\bar{\Lambda}_2,\\
    & && \hat{Q} = \hat{Q}^\top \succeq \mathbf{0}, \\
    & && \hat{R} = \hat{R}^\top \succeq \epsilon I, \\
    & && \hat{P} = \hat{P}^\top \succeq \epsilon I.
\end{aligned}
\end{equation}
The optimization problem is thus transformed into a feasibility problem. The objective is to find any set of variables that satisfies all the constraints, rather than minimizing a non-trivial objective functional. Since this problem involves linear equality constraints and definiteness constraints (i.e., conic constraints) with a linear (in fact, constant) objective function, it falls within the category of SDP and is specifically an SDP feasibility problem. Moreover, as all the constraints jointly define a convex feasible region and the objective function is convex, the entire problem is convex.

Next, we convert problem \eqref{eq:ioc_feasibility_form_2} into a solvable form. Let $\vecop(Z^\top)=Y\vecop(Z)$, where $Y$ is an $n^2\times n^2$ commutation matrix. The constraints can then be vectorized as follows.
\begin{equation} \label{eq:ioc_matrix}
    \begin{aligned}
 & \min_{Z, \hat{R}, \hat{Q}, \hat{P}, G}  &&  0  \\
& \text{\quad { } s.t.} &&  (Y + I) \vecop(Z) - ({K^*}^\top \otimes {K^*}^\top) \vecop(\hat{R}) + \vecop(\hat{Q}) = \mathbf{0}, \\
    &  && Y \vecop(Z) - ({K^*}^\top \otimes {K^*}^\top) \vecop(\hat{R}) - \vecop(G) = \mathbf{0}, \\
    &  && (\bar{\Lambda}^\top_1 \otimes I) \vecop(G) - (\bar{\Lambda}_2^\top \otimes I) \vecop(\hat{P}) = \mathbf{0}, \\
    & && \hat{Q}= \hat{Q}^\top \succeq \mathbf{0}, \\
    & && \hat{P} = \hat{P}^\top \succeq \epsilon I, \\
    & && \hat{R}= \hat{R}^\top \succeq \epsilon I,
\end{aligned}
\end{equation}
where the first three constraints are linear equalities. We define the stacked vector of decision variables as
\[\xi = \begin{bmatrix} \vecop(Z)^\top & \vecop(\hat{R})^\top & \vecop(\hat{Q})^\top & \vecop(\hat{P})^\top & \vecop(G)^\top \end{bmatrix}^\top.
\]
The equality constraints can therefore be expressed in a compact linear system form
\begin{equation} \label{eq:linear_system_form}
\underbrace{
    \begin{bmatrix}
        Y + I & - ({K^*}^\top \otimes {K^*}^\top) & I & \mathbf{0} & \mathbf{0} \\
       Y           & -({K^*}^\top \otimes {K^*}^\top)   & \mathbf{0} & \mathbf{0} & -I \\
        \mathbf{0}                      & \mathbf{0}                             & \mathbf{0} & -(\bar{\Lambda}_2^\top \otimes I) & (\bar{\Lambda}_1^\top \otimes I)
    \end{bmatrix}
}_{\displaystyle =:\Omega}
\xi = \mathbf{0}.
\end{equation}

To efficiently solve the SDP feasibility problem established in \eqref{eq:ioc_matrix}, we first equivalently reformulate it as a least-squares optimization problem with convex cone constraints
\begin{equation}\label{eq:ioc_LeastSquares} 
\begin{aligned}
& \min_{\xi}  && \lVert \Omega \xi \rVert_2^2 \\
& \text{ s.t.} && \vecop(\hat{Q})\in \Gamma_{+}^n,  \\
         & && \vecop(\hat{P})-\vecop(\epsilon I) \in \Gamma_{+}^n,  \\
         & && \vecop(\hat{R})-\vecop(\epsilon I) \in \Gamma_{+}^m,
\end{aligned}
\end{equation}
where the objective function minimizes the norm defined by the coefficient matrix $\Omega$ and the decision vector $\xi$. Here, $\Omega$ is the constant matrix constructed from expert data as defined in \eqref{eq:linear_system_form}. $\Gamma_{+}^k\doteq \{ \vecop(A) : A \in \mathbb{S}_{+}^k \}$ is the convex cone of vectorized $k \times k$ symmetric positive semidefinite matrices.
\begin{remark}
    The reformulation of the feasibility problem \eqref{eq:ioc_matrix} into the least-squares optimization problem \eqref{eq:ioc_LeastSquares} is a standard technique aimed at obtaining a smooth, differentiable objective function from a set of linear equality constraints. A solution to \eqref{eq:ioc_matrix} exists if and only if the optimal value of \eqref{eq:ioc_LeastSquares} is zero. However, since problem \eqref{eq:ioc_matrix} is a standard SDP feasibility problem, it could also be solved using existing solvers. The approach we adopt here is chosen as it provides the foundation for designing an efficient iterative solver later in the paper, while also enabling an explicit convergence rate analysis.
\end{remark}
\begin{lemma}[Lagrangian and Optimality Conditions]\label{lem:lagrangian}
Let $\vecop(M) = U_{M} \xi$, where the matrix $U_M$ is the selection matrix that serves to extract the sub-vector $\vecop(M)$ from the concatenated vector of all decision variables $\xi$. For the optimization problem \eqref{eq:ioc_LeastSquares}, the Lagrangian function $L(\xi; \lambda)$ is defined as
\begin{equation}
L(\xi;\lambda) = \lVert \Omega \xi \rVert_2^2 + \langle \lambda_{\hat{Q}}, U_{\hat{Q}} \xi \rangle + \langle \lambda_{\hat{P}}, U_{\hat{P}} \xi-\vecop(\epsilon I) \rangle + \langle \lambda_{\hat{R}}, U_{\hat{R}} \xi -\vecop(\epsilon I)\rangle,
\end{equation}
where $\lambda = (\lambda_{\hat{Q}}, \lambda_{\hat{P}}, \lambda_{\hat{R}})$ are the Lagrange multipliers for the three positive semidefinite cone constraints. Then, the optimal primal variable $\xi^*$ and dual variable $\lambda$ are related by
\begin{equation} \label{eq:xi_star_relation}
\xi^*(\lambda) = - \frac{1}{2} (\Omega^\top \Omega)^{\dagger} U^\top \lambda.
\end{equation}
\end{lemma}

\begin{proof}
According to the Karush-Kuhn-Tucker (KKT) optimality conditions, at the optimal solution, the gradient of the Lagrangian with respect to $\xi$ must be zero. Thus,
\begin{align*}
    \nabla_{\xi} L(\xi; \lambda) &= \frac{\partial}{\partial \xi} \left( \xi^\top \Omega^\top \Omega \xi + \lambda^\top U \xi -\lambda^\top W\right) \\
    &= 2 \Omega^\top \Omega \xi + U^\top \lambda = \mathbf{0},
\end{align*}
where $U$ is a concatenated projection matrix and $W$ is a constant vector, defined respectively as
\begin{equation}
    U = \begin{bmatrix} U_{\hat{Q}}^\top & U_{\hat{P}}^\top & U_{\hat{R}}^\top\end{bmatrix}^\top,\quad W= \begin{bmatrix} \mathbf{0}_{1\times n^2} & [\vecop(\epsilon I_{n})]^\top & [\vecop(\epsilon I_{m})]^\top \end{bmatrix}^\top.\nonumber
\end{equation}
Therefore, the optimal primal variable $\xi^*$ satisfying the stationary condition is related to the dual variable $\lambda$ by $\xi^*(\lambda) = - \frac{1}{2} (\Omega^\top \Omega)^{\dagger} U^\top \lambda$. Here we conclude the argument.
\end{proof}

\begin{theorem}[Convex Quadratic-Conic Dual Problem]\label{thm:dual_problem}
The Lagrangian dual of the optimization problem \eqref{eq:ioc_LeastSquares} is equivalent to the following convex quadratic-conic programming problem
\begin{equation}\label{eq:dual_final}
   \begin{aligned}
&\min_{\lambda} && \frac{1}{4} \lambda^\top H \lambda +\lambda^\top W \\
&\text{ s.t.} && \lambda_{\hat{Q}} \in \Gamma_+^n,  \\
        & && \lambda_{\hat{P}} \in \Gamma_+^n,  \\
        & && \lambda_{\hat{R}} \in \Gamma_+^m,
\end{aligned} 
\end{equation}
where $H = U (\Omega^\top \Omega)^{\dagger} U^\top$ is a symmetric positive semidefinite matrix.
\end{theorem}

\begin{proof}
The Lagrangian dual function $g(\lambda)$ is defined as the infimum of $L(\xi; \lambda)$ with respect to $\xi$: $g(\lambda) = \inf_{\xi} L(\xi; \lambda)=L(\xi^*(\lambda); \lambda)$. Substituting $\xi^*(\lambda)$ from Lemma \ref{lem:lagrangian} into the Lagrangian yields
\begin{align*}
g(\lambda) &= (\xi^*(\lambda))^\top \Omega^\top \Omega \xi^*(\lambda) + \lambda^\top U \xi^*(\lambda) -\lambda^\top W\\
&= \left( - \frac{1}{2} (\Omega^\top \Omega)^{\dagger} U^\top \lambda \right)^\top \Omega^\top \Omega \left( - \frac{1}{2} (\Omega^\top \Omega)^{\dagger} U^\top \lambda \right) + \lambda^\top U \left( - \frac{1}{2} (\Omega^\top \Omega)^{\dagger} U^\top \lambda \right) -\lambda^\top W\\
&= \frac{1}{4} \lambda^\top U ((\Omega^\top \Omega)^{\dagger})^\top \Omega^\top \Omega (\Omega^\top \Omega)^{\dagger} U^\top \lambda - \frac{1}{2} \lambda^\top U (\Omega^\top \Omega)^{\dagger} U^\top \lambda  -\lambda^\top W\\
&= - \frac{1}{4} \lambda^\top \left( U (\Omega^\top \Omega)^{\dagger} U^\top \right) \lambda  -\lambda^\top W \\
& = - \frac{1}{4} \lambda^\top H \lambda -\lambda^\top W.
\end{align*}
The dual problem is to maximize the dual function $g(\lambda)$ subject to the constraint that the Lagrange multipliers lie in the corresponding dual cones. The dual cone of the positive semidefinite cone $\Gamma_+^k$ is itself. Thus, the dual problem is $\max_{\lambda \in \Gamma_+^n \times \Gamma_+^n \times \Gamma_+^m} - \frac{1}{4} \lambda^\top H \lambda -\lambda^\top W$, which is equivalent to minimizing $-g(\lambda)$, as stated in \eqref{eq:dual_final}. 
\end{proof}


\begin{remark}
Strong duality holds between the primal problem \eqref{eq:ioc_LeastSquares} and its dual \eqref{eq:dual_final}. This is because the primal problem is convex and satisfies Slater's condition, as a strictly feasible point can be readily constructed from the true system parameters, which are assumed to exist. The guarantee of strong duality validates our approach of solving the primal problem by tackling its more tractable dual formulation.
\end{remark}

At this stage, we have successfully transformed a complex SDP feasibility problem into a more tractable convex quadratic-conic programming problem \eqref{eq:dual_final}. This dual problem is a well-structured convex optimization problem. If $\lambda^*$ is an optimal solution to this dual problem, the optimal solution $\xi^*$ to the primal problem \eqref{eq:ioc_LeastSquares} can be recovered via \eqref{eq:xi_star_relation}. Consequently, the core task ahead is to efficiently solve the dual problem \eqref{eq:dual_final} to obtain the optimal dual variables.

To facilitate the use of gradient-based optimization algorithms, we first convert this constrained dual problem into an equivalent, formally unconstrained optimization problem by introducing indicator functions. We incorporate the three semidefinite cone constraints as penalty terms in the objective, yielding a new objective function $J_{\text{dual}}(\lambda)$. The dual problem \eqref{eq:dual_final} is thereby equivalent to the unconstrained minimization problem
\begin{equation} \label{eq:dual_unconstrained}
\min_{\lambda} \quad J_{\text{dual}}(\lambda) = \frac{1}{4} \lambda^\top H \lambda +\lambda^\top W + I_{\Gamma_+^n}(\lambda_{\hat{Q}}) + I_{\Gamma_+^n}(\lambda_{\hat{P}}) + I_{\Gamma_+^m}(\lambda_{\hat{R}}).
\end{equation}
This objective function $J_{\text{dual}}(\lambda)$ is a non-smooth convex function. Given that its decision variable $\lambda$ is naturally partitioned into three blocks, namely $(\lambda_{\hat{Q}}, \lambda_{\hat{P}}, \lambda_{\hat{R}})$, the problem is particularly amenable to solution via Block Coordinate Descent (BCD) algorithms. The core idea of such algorithms is to decompose the complex joint optimization into a series of simpler subproblems. Specifically, in our context, this involves cyclically minimizing the objective function with respect to one block of variables (e.g., $\lambda_{\hat{Q}}$) while holding the others constant.

\begin{lemma}\label{lemma_4}
For the optimization problem~\eqref{eq:dual_unconstrained}, the subproblems with respect to each block of variables $(\lambda_{\hat{Q}}, \lambda_{\hat{P}}, \lambda_{\hat{R}})$ admit closed-form solutions.
\end{lemma}
\begin{proof}
    Due to structural similarity, we only prove the existence of a closed-form solution for the $\lambda_{\hat{Q}}$ subproblem, given fixed $\lambda_{\hat{P}}^*$ and $\lambda_{\hat{R}}^*$.
    We first partition the Hessian matrix $H$ conformally with $\lambda$ as
\begin{equation*}
H = \begin{bmatrix}
H_{QQ} & H_{QP} & H_{QR} \\
H_{QP}^\top & H_{PP} & H_{PR} \\
H_{QR}^\top & H_{PR}^\top & H_{RR}
\end{bmatrix}.
\end{equation*}
The partial subgradient of $J_{\text{dual}}$ with respect to $\lambda_{\hat{Q}}$ is then given by
\begin{equation}
\partial_{\lambda_{\hat{Q}}} J_{\text{dual}}(\lambda) = \frac{1}{2} (H_{QQ} \lambda_{\hat{Q}} + H_{QP} \lambda_{\hat{P}} + H_{QR} \lambda_{\hat{R}})+\partial I_{\Gamma_+^n}(\lambda_{\hat{Q}}).
\end{equation}
When updating $\lambda_{\hat{Q}}$ at iteration $k$, we fix $\lambda_{\hat{P}}^k$ and $\lambda_{\hat{R}}^k$. In the context of the BSUM framework, which is a variant of the BCD method, the update takes the form of a proximal step
\[
\lambda_{\hat{Q}}^{k+1} = \text{argmin}_{\lambda_{\hat{Q}}} \left( J_{\text{dual}}(\lambda_{\hat{Q}}, \lambda_{\hat{P}}^k, \lambda_{\hat{R}}^k) + \frac{1}{2} \lVert \lambda_{\hat{Q}} - \lambda_{\hat{Q}}^k  \rVert_{S_{\hat{Q}}}^2\right),
\]
where $S_{\hat{Q}}$ is a carefully chosen positive semidefinite matrix to ensure the subproblem objective for the $\lambda_{\hat{Q}}$-update is strongly convex. A valid choice is $S_{\hat{Q}} = \alpha I - \frac{1}{2} H_{QQ}$ with $\alpha=\lambda_{\text{max}}(H_{QQ})$. Applying the first-order optimality condition, we obtain
\begin{align*}
    \mathbf{0} & \in \frac{1}{2} (H_{QQ} \lambda_{\hat{Q}} + H_{QP} \lambda_{\hat{P}}^k + H_{QR} \lambda_{\hat{R}}^k)+\partial I_{\Gamma_+^n}(\lambda_{\hat{Q}})+\mathcal{S}_{\hat{Q}}(\lambda_{\hat{Q}} - \lambda_{\hat{Q}}^k)\\
    & = \alpha\lambda_{\hat{Q}}+\partial I_{\Gamma_+^n}(\lambda_{\hat{Q}})-\alpha\lambda_{\hat{Q}}^k+\frac{1}{2} (H_{QQ}\lambda_{\hat{Q}}^k+H_{QP} \lambda_{\hat{P}}^k + H_{QR} \lambda_{\hat{R}}^k).
\end{align*}
We define an intermediate variable $\Delta_{\hat{Q}}^k$ as
\begin{equation*}
\Delta_{\hat{Q}}^k = \lambda_{\hat{Q}}^k-\frac{1}{2\alpha} (H_{QQ}\lambda_{\hat{Q}}^k+H_{QP} \lambda_{\hat{P}}^k + H_{QR} \lambda_{\hat{R}}^k).
\end{equation*}
The final update rule can then be expressed in terms of a resolvent operator as
\begin{equation} \label{eq:update_rule_image}
\lambda_{\hat{Q}}^{k+1} = (I + \frac{1}{\alpha}\partial
I_{\Gamma_+^n})^{-1} \Delta_{\hat{Q}}^k,\quad \alpha=\lambda_{\text{max}}(H_{QQ}).
\end{equation}
Following similar steps, the updates for $\lambda_{\hat{P}}^{k+1}$ and $\lambda_{\hat{R}}^{k+1}$ are obtained by
\begin{equation} \label{eq:update_rule_image_2}
\begin{aligned}
    \lambda_{\hat{P}}^{k+1}& = (I + \frac{1}{\beta}(\vecop(\epsilon I)+\partial
I_{\Gamma_+^n}))^{-1} \Delta_{\hat{P}}^k,\quad \beta=\lambda_{\text{max}}(H_{PP}),\\
\lambda_{\hat{R}}^{k+1} &= (I +\frac{1}{\gamma}(\vecop(\epsilon I)+\partial
I_{\Gamma_+^m}))^{-1} \Delta_{\hat{R}}^k, \quad \gamma=\lambda_{\text{max}}(H_{RR}).
\end{aligned}
\end{equation}
Therefore, each subproblem of \eqref{eq:dual_unconstrained} admits a closed-form solution.
\end{proof}

\begin{remark}[On the Choice of BSUM]
    Lemma \ref{lemma_4}, which demonstrates the existence of closed-form solutions for each subproblem, is the core justification for choosing an algorithm based on the BCD philosophy, such as BSUM. Compared to full gradient descent methods, BCD-type algorithms can better exploit the separable structure of the problem. They solve subproblems via low-cost analytical updates rather than numerical iterations, leading to significant gains in computational efficiency. Although the convergence analysis of BCD-type algorithms is generally more complex than that of full gradient methods, the work \cite{hong2017iteration} provides a solid theoretical foundation that enables us to precisely analyze the convergence rate, which is crucial for guaranteeing the algorithm's reliability.
\end{remark}

\begin{remark}[Case of Multiple Trajectories]
    Suppose we can observe trajectories originating from multiple different initial states $x_1(0), \dots, x_l(0)$ with $l > n$. Let $\bar{X}(t) = [x_1(t), \dots, x_l(t)]$, which satisfies the dynamics $\dot{\bar{X}}=A_K^*\bar{X}$, where $A_K^*=A-BK^*$. If $\bar{X}$ satisfies a persistent excitation condition (i.e., $\bar{X}(t)$ is full row rank for some $t$), then $A_K^*$ can be uniquely determined by $A_K^*=\dot{\bar{X}}\bar{X}^\dagger$. In this scenario, the optimization problem \eqref{eq:ioc_feasibility_form_2} simplifies to the feasibility problem
\begin{equation} \label{eq:ioc_feasibility_form_3}
\begin{aligned}
&\min_{Z, \hat{Q}, \hat{R}, \hat{P}} && 0 \\
&\text{\quad s.t.} && Z^\top + Z - {K^*}^\top \hat{R} K^* + \hat{Q} = \mathbf{0}, \\
& &&Z^\top - {K^*}^\top \hat{R} K^* = \hat{P} A_K^*,\\
    & &&\hat{Q} = \hat{Q}^\top \succeq \mathbf{0}, \quad \hat{R} = \hat{R}^\top \succ \mathbf{0}, \quad \hat{P} = \hat{P}^\top \succ \mathbf{0}.
\end{aligned}
\end{equation}
This formulation is simpler as it contains one fewer equality constraint, making it easier to solve.
\end{remark}

\subsection{Algorithm and Convergence Analysis}

In the preceding section, we transformed the original IOC problem into the task of solving the Lagrangian dual problem \eqref{eq:dual_unconstrained}, which is a convex optimization problem. To solve this problem, we employ the BSUM algorithm \cite{hong2017iteration}. This method leverages the structure of objective function $J_{\text{dual}}(\lambda)$ with respect to the dual variable blocks $\lambda_{\hat{Q}}, \lambda_{\hat{P}}, \lambda_{\hat{R}}$. It iteratively and cyclically updates each block of variables while keeping the others fixed until convergence. Specifically, each update step involves solving a quadratic programming subproblem constrained to the corresponding semidefinite cone. Once the sequence of dual variables $\{\lambda^k\}$ converges to an optimal solution $\lambda^*$, we can recover the optimal solution $\xi^*$ to the primal problem \eqref{eq:ioc_LeastSquares} using the relationship \eqref{eq:xi_star_relation} derived from the KKT conditions.

The detailed procedure of our proposed Model-Free Inverse Optimal Control (MFIOC) algorithm is presented in Algorithm \ref{alg:mfioc_bsum}.

\begin{algorithm}
\caption{Model-Free Inverse Optimal Control (MFIOC)}
\label{alg:mfioc_bsum}
\begin{algorithmic}[1] 
    \Require Dataset $\mathcal{D}=\{x^*(t),u^*(t)\}$, initial point $(\lambda_{\hat{Q}}^0, \lambda_{\hat{P}}^0, \lambda_{\hat{R}}^0)$, tolerance $\varepsilon > 0$.
    \For{$k = 0, 1, 2, \dots$}
        \State Update $\lambda_{\hat{Q}}^{k+1}, \lambda_{\hat{P}}^{k+1}, \lambda_{\hat{R}}^{k+1}$ via the closed-form rules derived from the subproblems (e.g., \eqref{eq:update_rule_image}\eqref{eq:update_rule_image_2}).
        \If{$\lVert\lambda^{k+1}-\lambda^k\rVert < \varepsilon$}
            \State Compute $\xi^*(\lambda^{k+1})$ using \eqref{eq:xi_star_relation}.
            \State \textbf{break}
        \EndIf
    \EndFor
    \State \Return $\xi^*$
\end{algorithmic}
\end{algorithm}

Having proposed an iterative algorithm to solve the dual problem \eqref{eq:dual_unconstrained}, a natural and critical question arises: Does this algorithm converge to an optimal solution? If so, what is its rate of convergence? An analysis of the algorithm's convergence properties not only provides theoretical guarantees but also helps in understanding its practical efficiency.

The following theorem answers this question. It provides an explicit rate bound on the convergence of the objective function value $J_{\text{dual}}(\lambda^k)$ to the optimal value $J_{\text{dual}}(\lambda^*)$ for the sequence of dual variables $\{\lambda^k\}$, with $\lambda^k = (\lambda_{\hat{Q}}^k, \lambda_{\hat{P}}^k, \lambda_{\hat{R}}^k)$ generated by the iterative updates. This result demonstrates that even if the objective function $J_{\text{dual}}(\lambda)$ is not strongly convex, the algorithm is guaranteed to converge at a sublinear rate of $\mathcal{O}(1/k)$.

\begin{theorem}\label{theorem_convergence}
    Let $\{\lambda^k\}:=\{(\lambda_{\hat{Q}}^k, \lambda_{\hat{P}}^k, \lambda_{\hat{R}}^k)\}$ be the sequence generated by the BSUM update rules (e.g., \eqref{eq:update_rule_image}\eqref{eq:update_rule_image_2}). Then, the sequence satisfies the inequality
    \[
    J_{\text{dual}}(\lambda^k)-J_{\text{dual}}(\lambda^*)\leq \frac{c}{\sigma}\frac{1}{k},\quad \forall k\geq 1,
    \]
    where
    \begin{align*}
        R &= \sup_{\lambda\in \mathcal{F}, \lambda^*\in \Gamma_+^*} \left\{ \lVert\lambda-\lambda^*\rVert \right\}, \text{ with } \mathcal{F} = \{ \lambda : J_{\text{dual}}(\lambda)\leq J_{\text{dual}}(\lambda^0) \}, \\
        \sigma &= 1/(9\lVert H \rVert R^2), \\
        c &= \max \{4\sigma-2, J_{\text{dual}}(\lambda^0)-J_{\text{dual}}(\lambda^*), 2\}.
    \end{align*}
    Here, $\Gamma_+ = \Gamma_+^n \times \Gamma_+^n \times \Gamma_+^m$ is the feasible set, and $\Gamma_+^*$ denotes the set of optimal solutions.
\end{theorem}
\begin{proof}
    Let us decompose the objective function as $J_{\text{dual}}(\lambda^k)=g(\lambda^k)+\sum_{i} h_i(\lambda_i^k)$, where $i \in \{\hat{Q},\hat{P},\hat{R}\}$. We define
    \begin{align*}
        g(\lambda^k) &= \frac{1}{4}(\lambda^k)^\top H \lambda^k+(\lambda^k)^\top W, \\
        h_{\hat{Q}}(\lambda_{\hat{Q}}^k) &= I_{\Gamma_+^n}(\lambda_{\hat{Q}}^k), \\
        h_{\hat{P}}(\lambda_{\hat{P}}^k) &= I_{\Gamma_+^n}(\lambda_{\hat{P}}^k), \\
        h_{\hat{R}}(\lambda_{\hat{R}}^k) &= I_{\Gamma_+^m}(\lambda_{\hat{R}}^k).
    \end{align*}
    Since the matrix $H$ is positive semidefinite, $g(\lambda)$ is a smooth convex function and its gradient is $\nabla g(\lambda)=\frac{1}{2} H \lambda+W$. The gradient is Lipschitz continuous, because for any $\lambda_1,\lambda_2$ we have
    \[
    \lVert\nabla g(\lambda_1)-\nabla g(\lambda_2)\rVert = \frac{1}{2}\lVert H(\lambda_1-\lambda_2) \rVert \leq \frac{1}{2}\lVert H \rVert \cdot \lVert \lambda_1-\lambda_2 \rVert.
    \]
    Thus, $\nabla g(\lambda)$ has a Lipschitz constant of $M=\frac{1}{2}\lVert H \rVert$. Furthermore, $h_i(\lambda_i)$ is the indicator function of a semidefinite cone, which is a proper, closed, and convex function.
    The structure of our problem and the BSUM algorithm satisfy all the assumptions of Theorem 2 in \cite{hong2017iteration}. Therefore, the stated convergence rate is a direct corollary of this theorem.
\end{proof}

\begin{remark}[Interpretation of the Convergence Rate]
    Theorem \ref{theorem_convergence} proves that our algorithm possesses a sublinear convergence rate of $\mathcal{O}(1/k)$. It is crucial to emphasize that this theoretical rate describes a worst-case upper bound. In practice, as demonstrated in the subsequent simulation experiments (cf. Fig.~\ref{fig:my_ioc_convergence}), the algorithm's convergence speed is often significantly faster than this theoretical prediction, frequently exhibiting a near-linear trend. The sublinear rate in theory arises from the convexity, but not strong convexity, of the objective function, which is determined by the positive semidefiniteness of the matrix $H$. Nonetheless, having an explicit, non-asymptotic convergence rate bound is a core theoretical contribution of this work. It provides a quantifiable guarantee on the precision achievable within a finite number of iterations, which is vital for deploying the algorithm in high-reliability applications.
\end{remark}

\section{Simulations}\label{sim}
To validate the effectiveness, convergence, and stability of the proposed MFIOC algorithm, this section presents a series of simulation experiments. We first demonstrate the core performance of the algorithm through a typical single-trajectory case study and compare it with an existing differentiable programming algorithm. Subsequently, we conduct Monte Carlo experiments to systematically evaluate the robustness and superiority of both algorithms under different system parameters.
\subsection{Algorithm Convergence Analysis}

To quantitatively evaluate the effectiveness and convergence properties of the proposed algorithm, this section presents a set of numerical simulations. We postulate a nominal LQR system as the ground-truth model to generate the optimal expert trajectory. The system matrices and cost functional weighting matrices are defined as follows:
\begin{equation}
    \nonumber
    \begin{aligned}
&A=\left[\begin{array}{ccc}
-0.650 & -0.109 & -0.066 \\
-0.109 & -0.995 & -0.093 \\
-0.066 & -0.093 & -0.733
\end{array}\right] ,
&B=\left[\begin{array}{cc}
0.650 & 0.343 \\
-0.215 & -0.319 \\
0.778 & -0.101
\end{array}\right], \\
&Q=\left[\begin{array}{ccc}
1.393 & 0.120 & 0.146 \\
0.120 & 3.559 & -1.960 \\
0.146 & -1.960 & 1.702
\end{array}\right] ,
&R=\left[\begin{array}{cc}
3.761 & -0.324 \\
-0.324 & 3.719
\end{array}\right].
\end{aligned}
\end{equation}
Using the parameters above and a given initial state $x(0)=\left[\begin{array}{lll} -0.746 & 1.231 & 0.548 \end{array}\right]^T$, we generate an expert trajectory $(x^*(\cdot), u^*(\cdot))$ over a time horizon of $T=8s$ with a sampling interval of $\Delta t=0.1s$. The core objective of this experiment is to evaluate whether our proposed algorithm (MFIOC) can learn an equivalent LQR model from the observed trajectory data and accurately reproduce the expert's behavior. The core performance metric is the Mean Squared Error (MSE) between the reconstructed trajectory and the expert trajectory.

First, based on the generated expert trajectory, we identify the corresponding ground-truth optimal state-feedback gain $K^*$ via least-squares:
\begin{equation}
    K^*=\left[\begin{array}{ccc}
    0.161 & -0.316 & 0.285 \\
    0.098 & -0.135 & 0.083
    \end{array}\right].\nonumber
\end{equation}
The convergence process of the core dual variable $\lambda$ in our proposed MFIOC algorithm is shown in Fig.~\ref{fig:my_ioc_convergence}.
\begin{figure}[!htbp]
    \centering
    \includegraphics[width=0.7\linewidth]{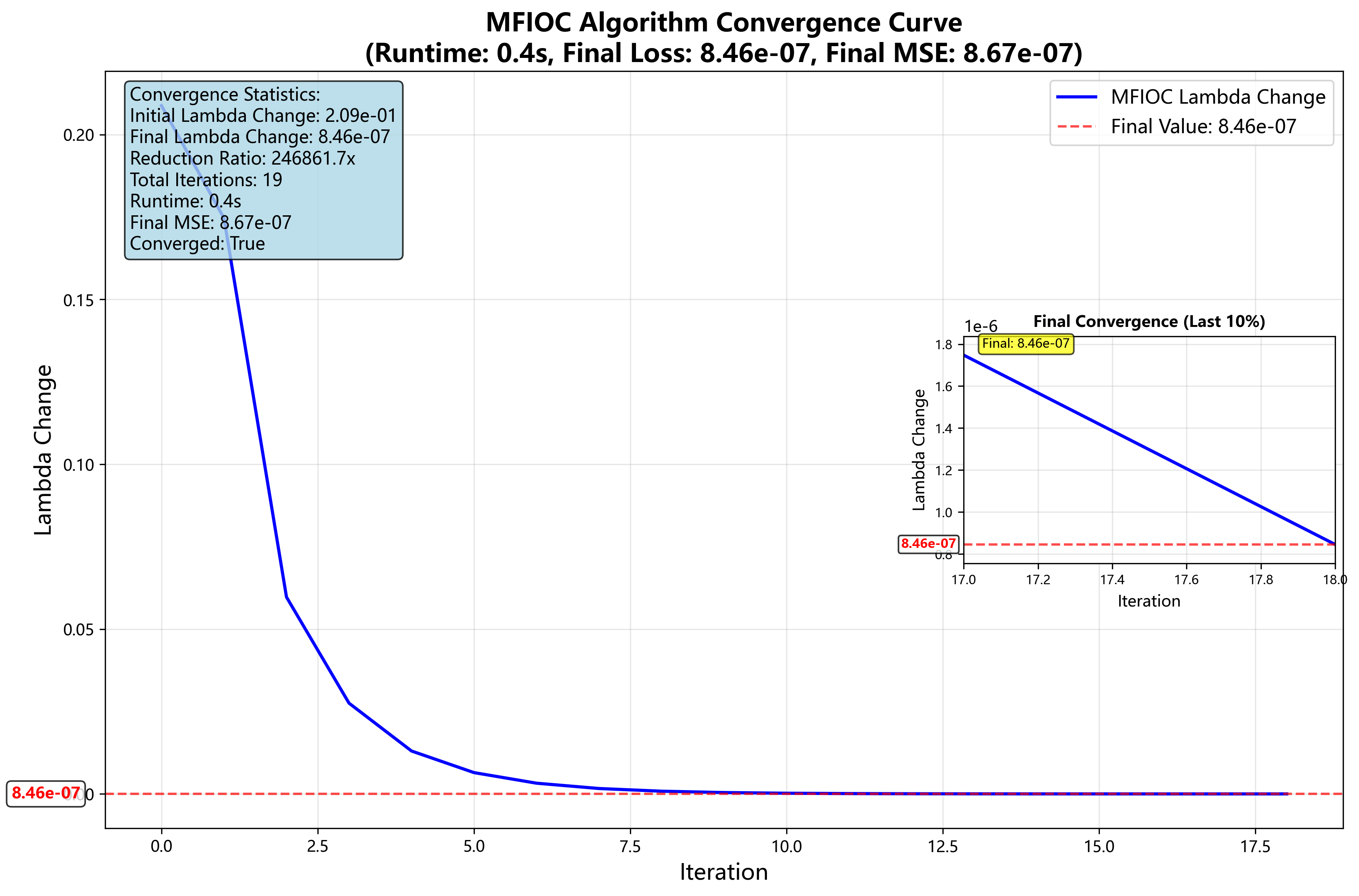}
    \caption{Convergence curve of the dual variable in our proposed algorithm.}
    \label{fig:my_ioc_convergence}
\end{figure}
As shown in Fig.~\ref{fig:my_ioc_convergence}, the algorithm exhibits extremely fast convergence, reaching a high-precision solution in only 19 iterations (0.3 seconds). This high computational efficiency stems from our core contribution of reformulating the original non-convex joint estimation problem into an equivalent convex second-order cone program. This transformation enables the design of an efficient solver based on the BSUM framework, where each subproblem has a closed-form solution, thereby avoiding the sensitivity to hyperparameters (e.g., step size) and costly line searches inherent in traditional gradient descent methods.

The optimal control gain recovered by the MFIOC algorithm is:
\begin{equation}
    \hat{K}_{\text{MFIOC}}=\left[\begin{array}{ccc}
     0.161 & -0.316 & 0.285 \\
    0.098 & -0.135 & 0.082
    \end{array}\right].\nonumber
\end{equation}
The Frobenius norm of the error, $||\hat{K}_{\text{MFIOC}} - K^*||_F$, is as low as $1.424 \times 10^{-4}$. The high accuracy in policy recovery validates the correctness of our theoretical framework. Since the algorithm is guaranteed to converge to an optimal solution within the feasible set, it can precisely identify an equivalent LQR model that perfectly explains the expert's behavior. This is corroborated by the trajectory reproduction shown in Fig.~\ref{fig:my_ioc_trajectory}, where the final Mean Squared Error (MSE) is only $8.67 \times 10^{-7}$.
\begin{figure}[!ht]
    \centering
    \includegraphics[width=0.7\linewidth]{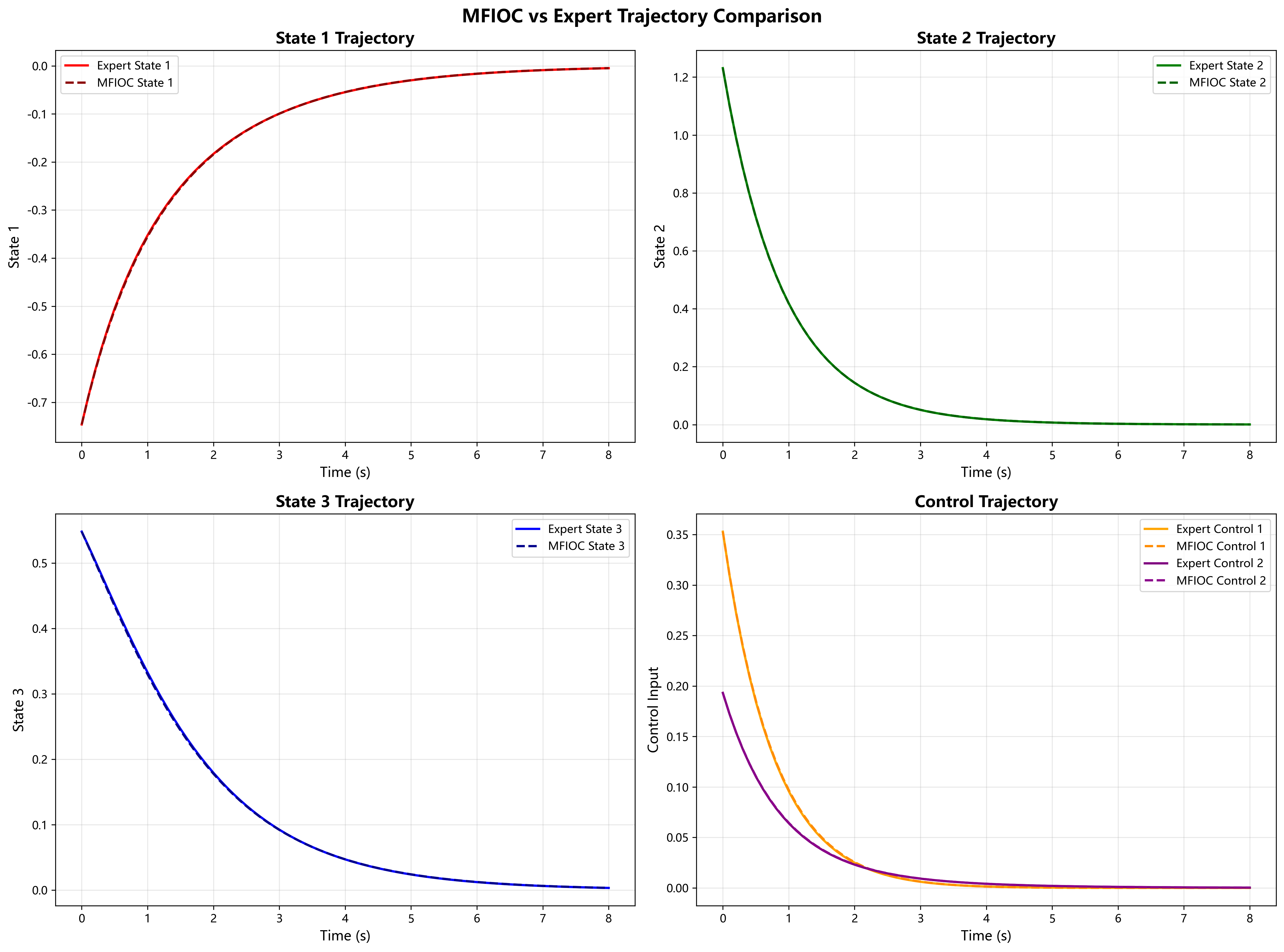}
    \caption{Comparison of the reconstructed trajectory (MFIOC) and the expert trajectory.}
    \label{fig:my_ioc_trajectory}
\end{figure}

To further highlight the characteristics of our method, we select the state-of-the-art PDP algorithm as a benchmark. As a general-purpose, end-to-end learning and control framework based on differentiable optimization, PDP's strength lies in its ability to handle complex nonlinear systems and non-quadratic cost functions without requiring special structural transformations, thus offering strong generality. 

However, this generality also presents challenges in the context of the inverse LQR problem studied herein. From a theoretical standpoint, PDP's slower convergence can be attributed to two main factors: (1) Non-Convex Optimization Landscape. The joint estimation of system and cost parameters constitutes a non-convex problem. Any first-order optimization algorithm based on local gradient information naturally faces challenges such as slow convergence and the risk of getting trapped in local minima. (2) High Cost per Iteration. PDP's gradient computation follows a "forward-backward" pass. Specifically, each iteration requires solving a full optimal control problem (the forward pass) and an auxiliary LQR problem (the backward pass). Both steps involve integration or recursion over the entire trajectory, and although the authors show that the complexity of the backward pass is only linear in the time horizon—a significant advantage over some alternatives—the per-iteration cost remains substantial compared to our MFIOC method. Our algorithm, through its convex reformulation, concentrates the main computational load into a one-time matrix construction phase (building $\Omega$ and $H$). Subsequent iterations are extremely lightweight, as they only involve closed-form updates of the dual variables, thus fundamentally improving computational efficiency.

The experimental results clearly corroborate this fundamental difference. As shown in Fig.~\ref{fig:pdp_convergence}, the high per-iteration cost and the exploration within a non-convex landscape led the PDP algorithm to take 2000 iterations and 98.4 seconds to converge. The final control gain it learned is
\begin{equation}
    \hat{K}_{\text{PDP}}=\left[\begin{array}{ccc}
     0.205 & -0.330 & 0.333 \\
    0.123 & -0.131 & 0.086
    \end{array}\right],\nonumber
\end{equation}
which exhibits a relative error of approximately $7.1 \times 10^{-2}$ and leads to perceptible deviations in trajectory reproduction (Fig.~\ref{fig:pdp_trajectory}).
\begin{figure}[!htbp]
    \centering
    \includegraphics[width=0.7\linewidth]{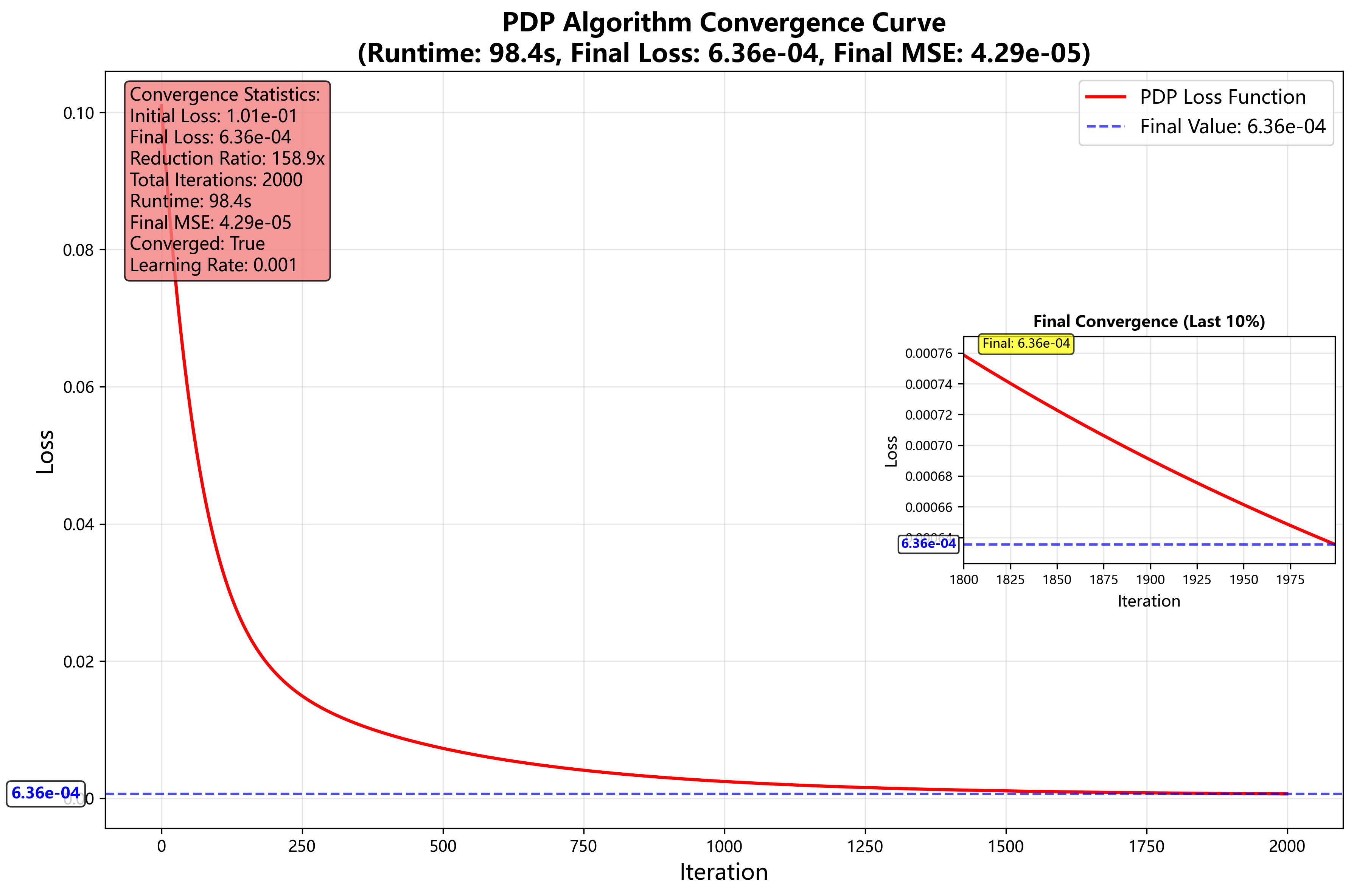}
    \caption{Convergence curve of the loss function for the PDP algorithm.}
    \label{fig:pdp_convergence}
\end{figure}
\begin{figure}[!htbp]
    \centering
    \includegraphics[width=0.7\linewidth]{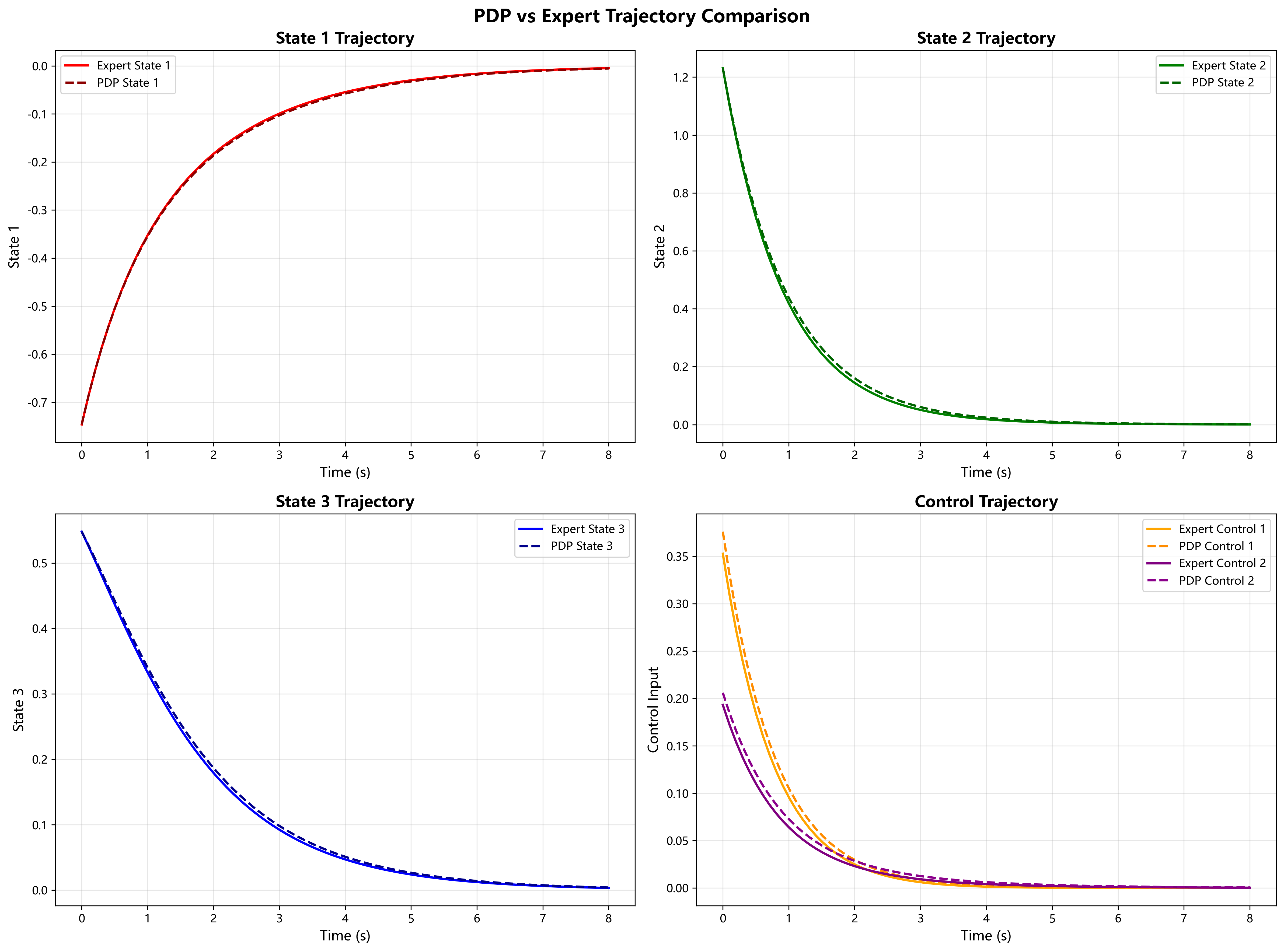}
    \caption{Comparison of the reconstructed trajectory (PDP) and the expert trajectory.}
    \label{fig:pdp_trajectory}
\end{figure}
The performance of the two algorithms is summarized in Table~\ref{tab:comparison}.
\begin{table}[h]
\centering
\caption{Performance Comparison between MFIOC and PDP}
\label{tab:comparison}
\begin{tabular}{lcc}
\hline
\textbf{Performance Metric} & \textbf{MFIOC (Proposed)} & \textbf{PDP (Benchmark)} \\ \hline
Runtime (s) & \textbf{0.3} & 98.4 \\
Iterations & \textbf{19} & 2000 \\
$K$ Relative Error & $\mathbf{1.424 \times 10^{-4}}$ & $7.1 \times 10^{-2}$ \\
Trajectory MSE & $\mathbf{8.67 \times 10^{-7}}$ & $4.29 \times 10^{-5}$ \\ \hline
\end{tabular}
\end{table}

This comparison clearly illustrates the trade-off between a \textit{customized solution path} and a \textit{general-purpose optimization framework}. Our MFIOC algorithm achieves superior computational efficiency and optimality guarantees by exploiting the inherent structure of the LQR problem—a "specialization" approach. In contrast, the PDP algorithm represents a "generalization" approach, maintaining applicability to a broader class of problems at the cost of weaker performance guarantees and lower efficiency when applied to this specific problem class.

\subsection{Algorithm Stability Analysis}
To comprehensively assess the stability and generalizability of the algorithm, we conducted 100 independent Monte Carlo experiments with $n=3$ and $m=2$. In each trial, a new set of stabilizable system parameters and a new initial state were randomly generated to produce an expert trajectory. The two algorithms are then tasked with tracking these 100 trajectories. The statistical comparison of their performance is presented in Fig.~\ref{fig:monte_carlo}.

\begin{figure}[!hbtp]
    \centering
    \includegraphics[width=0.7\linewidth]{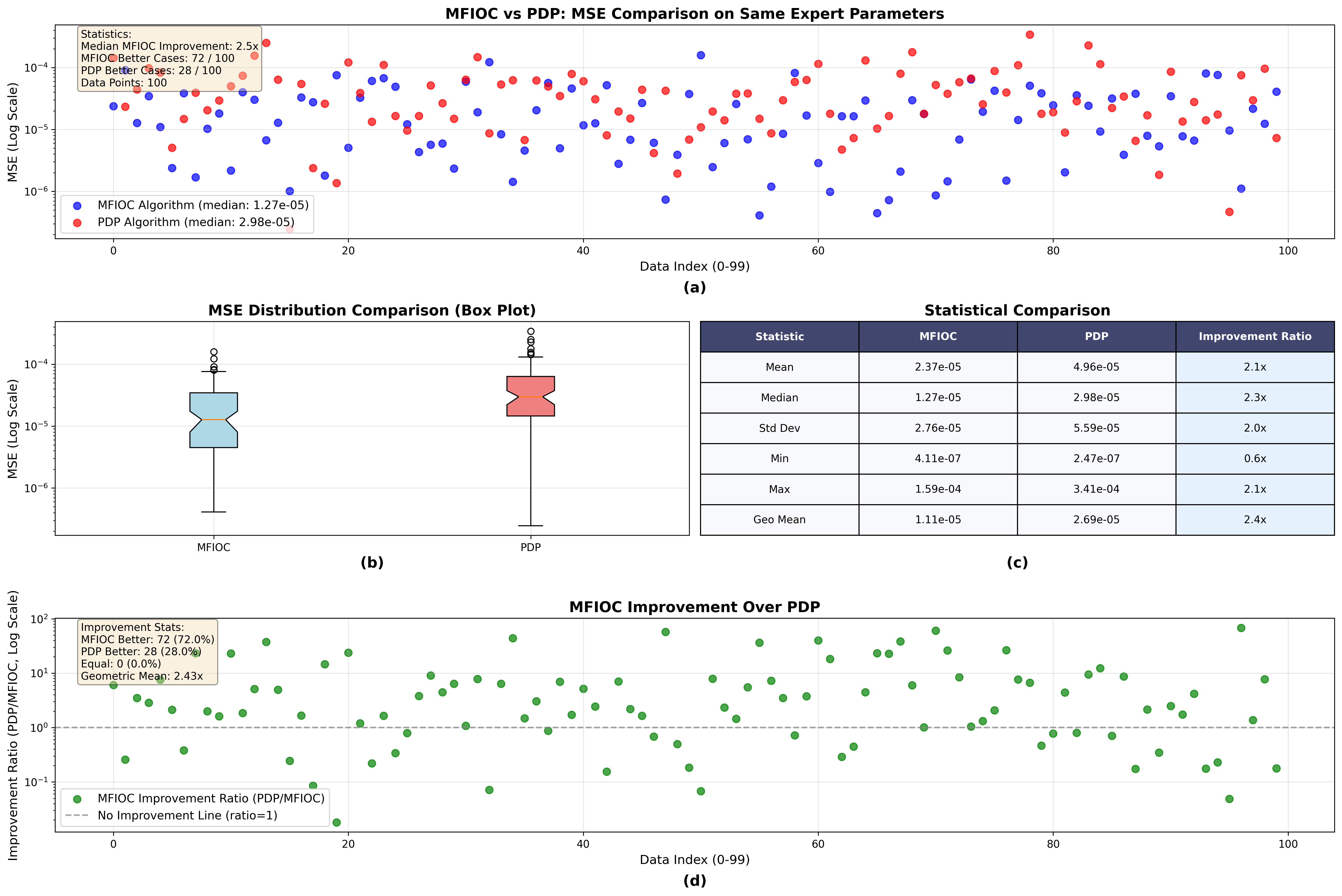}
    \caption{Performance comparison over 100 Monte Carlo experiments.}
    \label{fig:monte_carlo}
\end{figure}

Figure~\ref{fig:monte_carlo} provides an in-depth comparison of the MSE from multiple perspectives. Fig.~\ref{fig:monte_carlo}(a) clearly shows that in the vast majority of experiments, the MSE values for our algorithm (blue dots) are lower than those for the PDP algorithm (red dots), indicating a general accuracy advantage. The box plot in Fig.~\ref{fig:monte_carlo}(b) visually confirms that the MSE distribution for our algorithm is more concentrated and shifted downwards. The statistical table in Fig.~\ref{fig:monte_carlo}(c) provides quantitative support: our algorithm's median MSE ($1.27 \times 10^{-5}$) is about half that of the PDP algorithm ($2.98 \times 10^{-5}$), representing a 2.3-fold performance improvement. Our method also shows significant superiority in terms of mean, maximum error, and standard deviation, proving that it is not only more accurate but also more stable. Finally, Fig.~\ref{fig:monte_carlo}(d) displays the MSE improvement ratio of our algorithm relative to PDP in each experiment. The statistics show that our algorithm outperformed PDP in 72\% of the trials. These data strongly corroborate the robustness and superiority of our proposed algorithm.

In summary, the simulation results fully validate the effectiveness of our proposed algorithm. Compared to the state-of-the-art PDP algorithm, our method achieves an order-of-magnitude improvement in computational efficiency while also demonstrating significant advantages in reconstruction accuracy and stability. This indicates that our proposed theoretical framework, which transforms the inverse LQR problem into a convex conic program, is both efficient and reliable in practice.

\section{Conclusion}\label{conclusion}
This paper addressed the inverse LQR problem for systems with unknown dynamics and cost functionals by proposing a solution framework with rigorous convergence and convergence rate guarantees. To overcome the lack of theoretical assurances in existing methods, we first equivalently reformulated this non-convex inverse problem into a convex second-order cone program. Subsequently, by solving its Lagrangian dual using the BSUM algorithm, we designed an efficient solution procedure. Concurrently, we strictly proved that the proposed algorithm has an $\mathcal{O}(1/k)$ convergence rate, providing what is, to the best of our knowledge, the first solution for the model-free IOC problem with an explicit convergence rate guarantee. Simulation results validated the effectiveness of this framework: compared to an advanced differentiable programming algorithm, our method achieved a considerable improvement in computational efficiency and demonstrated significant advantages in reconstruction accuracy and robustness.

Future work can be extended in several directions. First, the current theoretical framework could be extended to discrete-time systems and output-feedback systems to accommodate a broader range of applications. Second, the robustness and performance bounds of the algorithm could be investigated when expert trajectories contain noise or are sub-optimal. Finally, exploring the applicability of this method in partially observable scenarios presents a valuable research avenue.

\bibliographystyle{unsrt}
\bibliographystyle{plain}        
\bibliography{reference}           





\end{document}